\documentclass[12pt]{amsart}
\usepackage{amssymb}
\usepackage{enumerate}

\newcommand{\N}{{\mathbb N}}
\newcommand{\C}{{\mathbb C}}

\newcommand{\Z}{{\mathbb Z}}

\newcommand{\wand}{wandering domain}

\newcommand{\tef}{transcendental entire function}

\newcommand{\nhd}{neighbourhood}
\newcommand{\sconn}{simply connected}
\newcommand{\mconn}{multiply connected}

\newcommand{\lconn}{locally connected}
\newcommand{\sw}{spider's web}

\theoremstyle{plain}
\newtheorem{theorem}{Theorem}[section]
\newtheorem{corollary}[theorem]{Corollary}
\newtheorem*{theorem*}{Theorem}
\newtheorem*{definition*}{Definition}
\newtheorem*{proposition*}{Proposition}

\newtheorem{proposition}[theorem]{Proposition}
\newtheorem{lemma}[theorem]{Lemma}
\theoremstyle{definition}

\theoremstyle{remark}
\newtheorem*{remark}{Remark}

\theoremstyle{problem}

\theoremstyle{example}
\newtheorem{example}[theorem]{Example}

\newenvironment{myindentpar}[1]%
{\begin{list}{}%
         {\setlength{\leftmargin}{#1}}%
         \item[]%
}
{\end{list}}

\begin{document}


\title[Spiders' Webs and Locally Connected Julia Sets]{Spiders' Webs and Locally Connected Julia Sets \\of Transcendental Entire Functions}

\author{J.W. Osborne}
\address{Department of Mathematics and Statistics \\
   The Open University \\
   Walton Hall\\
   Milton Keynes MK7 6AA\\
   UK}
\email{j.osborne@open.ac.uk}





\begin{abstract}
We show that, if the Julia set of a \tef\ is \lconn, then it takes the form of a \sw\ in the sense defined by Rippon and Stallard.  In the opposite direction, we prove that a \sw\ Julia set is always \lconn\ at a dense subset of buried points.  We also show that the set of buried points (the residual Julia set) can be a \sw.
\end{abstract}

\maketitle

\section{Introduction} 
\label{intro}
\setcounter{equation}{0}

Let $ f $ be a function which is either transcendental entire or rational of degree at least two, and let $ f^n, n = 0, 1, 2 \ldots $, denote the $ n $th iterate of $ f $.  The \textit{Fatou set} $ F(f) $ is defined to be the set of points $ z \in \C $ (or $ z \in \widehat{\C} = \C \cup \lbrace \infty \rbrace $ if $ f $ is rational) such that the family of functions $ \lbrace f^n : n \in \N \rbrace $ is normal in some \nhd\ of $ z $. The complement of $ F(f) $ is called the \textit{Julia set} $ J(f) $. For an introduction to iteration theory and the properties of these sets, see \cite{aB, CG, Mil} for rational maps and \cite{wB93, MNTU} for \tef s.

The Julia set $ J(f) $ often displays considerable geometric and topological complexity.  It is of interest to ask when $ J(f) $ is \lconn\ at some or all of its points, and what other properties then follow.  

Rational maps with \lconn\ Julia sets have been much studied, and several classes of functions are known for which, if the Julia set is connected, then it is also \lconn\ - see, for example, \cite[Chapter 19]{Mil} and \cite{CJY, MH, TY}.  Some analogous results have been obtained for \tef s \cite{BM02, Mor}, but the situation is less well understood.  More details are given in Section \ref{examples}.

In this paper, we explore some links between the local connectedness of $ J(f) $ for a \tef, and a particular geometric form of $ J(f) $ known as a \sw.  Following Rippon and Stallard \cite{RS10a}, we define a set $ E $ to be a \textit{spider's web} if $ E $ is connected, and there exists a sequence $ (G_k)_{k \in \N} $ of bounded, \sconn\ domains such that
\begin{equation}
\label{web}
G_k \subset G_{k+1}  \text{ and } \partial G_k \subset E, \text{ for each } k \in \N, \quad \text{ and } \bigcup_{k \in \N} G_k = \C.
\end{equation}
We sometimes refer to elements of the sequence $ (\partial G_k)_{k \in \N} $ as \textit{loops} in the \sw. 

Our main result is the following. 

\begin{theorem}
\label{lcsw}
Let $ f $ be a \tef\ such that $ J(f) $ is \lconn.  Then $ J(f) $ is a \sw.
\end{theorem}

In fact, we can show that $ J(f) $ is a \sw\ under weaker hypotheses than the local connectedness of $ J(f) $.  Details are given in Section \ref{Thm lc}.

An immediate corollary of Theorem \ref{lcsw} is the following. Here a \textit{Fatou component} is a component of $ F(f) $. 
\begin{corollary}
\label{unbounded}
Let $ f $ be a \tef\ with an unbounded Fatou component.  Then $ J(f) $ is not \lconn.
\end{corollary}

In \cite[Theorem E]{BD1}, Baker and Dom\'{i}nguez showed that, if a \tef\ $ f $ has an unbounded invariant Fatou component $ U $, then $ J(f) $ is not locally connected at any point, except perhaps in the case when $ U $ is a Baker domain and $ f \vert_U $ is univalent.  In this exceptional case it is possible for the boundary of $ U $ to be a Jordan arc \cite{BW91}, but Corollary \ref{unbounded} shows that, even then, $ J(f) $ cannot be \lconn\ at all of its points.

In the opposite direction to Theorem \ref{lcsw}, we prove the following result.  We say that a set $ S \subset \C $ \textit{surrounds} a point $ z $ if $ z $ lies in a bounded complementary component of $ S $.

\begin{theorem}
\label{rings}
Let $ f $ be a \tef\ such that $ J(f) $ is a \sw.  Then there exists a subset of $ J(f) $ which is dense in $ J(f) $ and consists of points $ z $ with the property that every \nhd\ of $ z $ contains a continuum in $ J(f) $ that surrounds ~$ z $.  Each such point $ z $ is a buried point of $ J(f) $ at which $ J(f) $ is \lconn.     
\end{theorem}

Recall that $ z \in J(f) $ is a \textit{buried point} if $ z $ does not lie on the boundary of any Fatou component. The set of all buried points is called the \textit{residual Julia set}, denoted by $ J_r(f). $  It follows from Theorem \ref{rings} that $ J_r(f) $ is never empty for a \tef\ $ f $ for which $ J(f) $ is a \sw.  

Using Theorem \ref{rings} and a topological result due to Whyburn (see Lemma \ref{whyburn2}), we can build on Theorem \ref{lcsw} to obtain more detailed properties of \lconn\ Julia sets for \tef s.

\begin{theorem}
\label{burjor}
Let $ f $ be a \tef\ such that $ J(f) $ is \lconn.  Then 
\begin{enumerate}[(a)]
\item $ J_r(f) \neq \emptyset $, and every \nhd\ of a point $ z \in J_r(f) $ contains a Jordan curve in $ J(f) $ surrounding $ z $;
\item $ J(f) $ is a \sw, and there exists a sequence $ (G_k)_{k \in \N} $ of bounded, \sconn\ domains satisfying (\ref{web}) with $ E = J(f) $, such that the loops $ (\partial G_k)_{k \in \N} $ are Jordan curves.
\end{enumerate} 
\end{theorem} 

The organisation of this paper is as follows.  In Section \ref{prelim}, we gather together for convenience a number of results used later in the paper, and establish some notation and definitions.  In Section \ref{Thm lc}, we prove Theorem \ref{lcsw} and related results, whilst in Section \ref{Thm rings} we prove Theorems \ref{rings} and \ref{burjor}.   In Section \ref{residual}, we give some results on the residual Julia set $ J_r(f) $, including the fact that there are classes of functions for which $ J_r(f) $ is itself a \sw.    Finally, in Section \ref{examples}, we give a number of examples to illustrate the results of previous sections.  We show that the Julia set of the function $ \sin z $ is a \sw, and give some new examples of \tef s for which the Julia set is \lconn.

\section{Preliminaries} 
\label{prelim}
\setcounter{equation}{0} 

In this section, we gather together a number of results that will be used later in the paper.  We also establish some notation and define certain terms.

We denote by $ B(a, r) $ the open disc $ \lbrace z : \vert z - a \vert < r \rbrace $, by $ \overline{B}(a, r) $ the corresponding closed disc and by $ C(a,r) $ the circle $ \lbrace z : \vert z - a \vert = r \rbrace $.

If $ S $ is a subset of $ \C $, we use the notation $ \widetilde{S} $ to denote the union of $ S $ and all its bounded complementary components (if any).  

A Hausdorff space $ X $ is \textit{\lconn} at the point $ x \in X $ if $ x $ has arbitrarily small connected (but not necessarily open) \nhd s in $ X $.  If this is true for every $ x \in X $, then we say that $ X $ is \lconn\ (see, for example, Milnor \cite[p. 182]{Mil}).

We will need the following topological results due to Whyburn.  Here a \textit{plane continuum} is a compact, connected set lying in the plane or on the Riemann sphere.

\begin{lemma} \cite[Ch.VI, (4.4)]{Why}
\label{whyburn} 
A plane continuum is locally connected if and only if
\begin{enumerate}[(a)]
\item the boundary of each of its complementary components is locally connected, and 
\item for each $ \varepsilon > 0 $, at most finitely many of these complementary components have spherical diameters greater than $ \varepsilon $. 
\end{enumerate} 
\end{lemma}

\begin{lemma} \cite[Ch.VI, (4.5)]{Why} 
\label{whyburn2} 
If a point $ p $ in a \lconn\ plane continuum $ E $ is not on the boundary of any complementary component of $ E $, then for each $ \varepsilon > 0 $, $ E $~contains a Jordan curve of spherical diameter less than $ \varepsilon $ surrounding $ p $.
\end{lemma}

We make use of the following results on the connectedness properties of the Julia set of a \tef, due to Kisaka and to  Baker and Dom\'{i}nguez.  

\begin{lemma} \cite[Theorem 2]{K1}
\label{kisaka} 
If $ f $ is a \tef\ such that all components of $ F(f) $ are bounded and \sconn, then $ J(f) $ is connected. 
\end{lemma}

\begin{lemma} \cite[part of Theorem A]{BD1}
\label{bakerdom} 
If $ f $ is a \tef\ such that $ J(f) $ is \lconn\ at one of its points, then $ J(f) $ is connected.
\end{lemma}

\begin{lemma} \cite[Corollary 3]{BD1}
\label{bakerdom2} 
If $ f $ is a \tef\ and $ F(f) $ has a completely invariant component, then $ J(f) $ is not \lconn\ at any point.  
\end{lemma}

We also recall some of the terminology associated with the dynamics of \tef s.

If $ U = U_0 $ is a Fatou component, then for each $ n \in \N $, $ f^n(U) \subset U_n $ for some Fatou component $ U_n $.  If $ U = U_n $ for some $ n \in \N $, we say that $ U $ is \textit{periodic} or \textit{cyclic}, and if $ n = 1 $, that $ U $ is \textit{invariant}. If $ U $ is not eventually periodic, i.e. if $ U_m \neq U_n $ for all $ n > m \geq 0 $, then $ U $ is called a \textit{wandering} Fatou component or a \textit{\wand}.  Periodic Fatou components for \tef s can be classified into four types, namely immediate attracting basins, immediate parabolic basins, Siegel discs and Baker domains.  We refer to \cite{wB93, MNTU} for the definitions and properties of such components. 

Note that, if $ U $ is a Fatou component such that $ f^{-1}(U) \subset U $, then it follows that $ f(U) \subset U $, and $ U $ is referred to as \textit{completely invariant}.  It is shown in \cite{iB70} that, if $ f $ is a \tef, there can be at most one such component.

The \textit{exceptional set} $ E(f) $ is the set of points with a finite backwards orbit under ~$ f $. For a \tef\ $ E(f) $ contains at most one point.  We will need the well-known \textit{blowing up property} of the Julia set $ J(f) $:

\begin{myindentpar}{1cm}
if $ f $ is an entire function, $ K $ is a compact set, $ K \subset \C \setminus E(f) $ and $ V $ is an open \nhd\ of $ z \in J(f) $, there exists $ N \in \N $ such that $ f^n(V) \supset K $, for all $ n \geq N. $
\end{myindentpar}

The dynamical behaviour of a \tef\ is much affected by the properties of its set of \textit{singular values}, that is, the set of all of its critical values and finite asymptotic values.  The set of singular values of $ f $ is denoted by $ \text{sing}(f^{-1})$, and $ f $ is said to be in the Speiser class $ \mathcal{S} $ if $ \text{sing}(f^{-1}) $ is a finite set, and in the Eremenko-Lyubich class $ \mathcal{B} $ if $ \text{sing}(f^{-1}) $ is bounded. 
 
Finally in this section, we give a definition and a result from Rippon and Stallard's paper \cite{RS10a} where the notion of a \sw\ was first introduced.

Let $ I(f) $ denote the set of points whose orbits escape to infinity,
\[ I(f) = \lbrace z \in \C : f^n(z) \to \infty \text{ as } n \to \infty \rbrace.\]
Then we define the subset $ A_R(f) $ of $ I(f) $ as follows.  Let $ R > 0 $ be such that $ M(r, f) > r $ for $ r \geq R $.  Then
\[ A_R(f) = \lbrace z \in \C : \vert f^n(z) \vert \geq M^n(R, f), \text{ for } n \in \N \rbrace, \]  
where $ M(r, f) = \max_{\vert z \vert = r} \vert f(z) \vert \text{ for } r > 0, $ and $ M^n(r,f) $ denotes the $ n $th iterate of $ M $ with respect to $ r $.  

The result we will need from \cite{RS10a} is the following.

\begin{lemma} \cite[Theorem 1.5]{RS10a} 
\label{fast}
Let $ f $ be a \tef, let $ R > 0 $ be such that $ M(r, f) > r $ for $ r \geq R $, and let $ A_R(f) $ be a \sw.  If $ f $ has no \mconn\ Fatou components, then $ J(f) $ is a \sw.
\end{lemma}

\section{Proof of Theorem \ref{lcsw} and related results} 
\label{Thm lc}
\setcounter{equation}{0} 

We now prove the following result and show that this implies Theorem \ref{lcsw}.

\begin{theorem}
\label{no unbounded}
Let $ f $ be a \tef\ such that:
\begin{enumerate}[(1)]
\item $ F(f) $ has no completely invariant component, and
\item for each $ \varepsilon > 0 $, at most finitely many components of $ F(f) $ have spherical diameters greater than $ \varepsilon $. 
\end{enumerate}  
Then $ F(f) $ has no unbounded components, and there exists a sequence $ (G_k)_{k \in \N} $ of bounded, \sconn\ domains such that
\begin{itemize}
\item $ G_{k+1} \supset G_k $, for $ k \in \N $,
\item $ \partial G_k \subset J(f) $, for $ k \in \N $ and
\item $ \bigcup_{k \in \N} G_k = \C. $
\end{itemize}
\end{theorem}

\begin{corollary}
\label{spider}
Let $ f $ be a \tef\ satisfying the assumptions of Theorem \ref{no unbounded}, and assume further that $ F(f) $ has no \mconn\ components.  Then $ J(f) $ is a \sw.
\end{corollary}

Note that, if $ J(f) $ is \lconn, it follows from Lemmas \ref{whyburn} and \ref{bakerdom2} that the assumptions of Theorem \ref{no unbounded} hold. Theorem \ref{lcsw} then follows because $ J(f) $ is connected by Lemma \ref{bakerdom}.

We will need the following simple lemma, in which by a \textit{preimage} of a Fatou component $ U $, we mean a component of $ f^{-n}(U) $ for some $ n \in \N $.  This result is surely known, but we include a proof for completeness as we have been unable to locate a reference.  
\begin{lemma}
\label{infinite}
Let $ f $ be a \tef. Then every component of $ F(f) $ which is not completely invariant has infinitely many distinct preimages. 
\end{lemma}

\begin{proof}
Assume, for a contradiction, that $ U $ is a component of $ F(f) $ which is not completely invariant and has only finitely many distinct preimages. 

We first show $ U $ must be periodic.  For suppose that $ U $ is non-periodic, and let $ U' $ be any preimage of $ U $.  Then $ U' $ is a component of $ f^{-n}(U) $ for some $ n \in \N $, and since $ U $ is not periodic, there must be at least one component of $ f^{-1}(U') $ which is distinct from every component of $ f^{-k}(U) $ for $ k = 1, \ldots, n. $  As this is true for all components of $ f^{-n}(U) $ and for every $ n \in \N $, $ U $ must have infinitely many distinct preimages, which is against our assumption. 

Thus $ U $ must belong to some cycle of period $ p > 1 $ in which no element in the cycle has preimages outside the cycle. But then each of the $ p $ distinct elements in the cycle is a completely invariant component of $ F(f^p) $, contradicting the fact that a \tef\ can have at most one completely invariant Fatou component \cite{iB70}.  This contradiction completes the proof.   
\end{proof}
 
\begin{proof}[Proof of Theorem \ref{no unbounded}]

We first suppose that $ F(f) $ has an unbounded component ~$ V $, and seek a contradiction.  

Let $ R>0 $ be so large that $ B(0,R) \cap J(f) \neq \emptyset $. Then we claim that infinitely many preimages of $ V $ must meet $ B(0,R). $  

To show this, note that $ V $ has infinitely many distinct unbounded preimages, by Lemma \ref{infinite}.  Thus, if only finitely many of these preimages meet $ B(0,R) $, there must be some preimage $ W $ of $ V $ that is not periodic and does not meet $ B(0,R). $  But since $ B(0,R) \cap J(f) \neq \emptyset $, it follows from the blowing up property of $ J(f) $ that there exists $ N \in \N $ such that $ f^n(B(0,R)) $ meets $ W $ for all $ n \geq N. $  Thus we may choose a strictly increasing sequence $ (n_j)_{j \in \N} $ of integers greater than $ N $ such that some component $ X_{n_j} $ of $ f^{-n_j}(W) $ meets $ B(0,R) $ for all $ j \in \N $.  

Now suppose that two such components coincide.  Then there exist $ k, l \in \N $ with $ k > l $ (and thus $ n_k > n_l $), and $ X_{n_k} = X_{n_l} = X, $ say, such that
\[ f^{n_l}(X) \subset W \]
and
\[ f^{n_k}(X) \subset W. \]
It then follows that $ f^{n_k - n_l}(W) \subset W $, so that $ W $ is periodic, contrary to our assumption.  This proves the claim.

Now since every point on the circle $ C(0,R) $ lies at the same spherical distance from $ \infty $, the fact that infinitely many unbounded preimages of $ V $ meet $ B(0,R) $ contradicts property (2) in the statement of Theorem \ref{no unbounded}.  Thus it follows that there are no unbounded components of $ F(f) $.

Now let $ r > 0 $, and let $ \varepsilon > 0 $ be less than the spherical distance of the circle $ C = C(0,r) $ from $ \infty $.  Let $ \lbrace U_j : j \in \N \rbrace $ be the collection of components of $ F(f) $ that meet $ C $.  This collection may be empty, finite or countably infinite, but
\begin{enumerate}[(i)]
\item we have just proved that none of the $ U_j $ is unbounded, and
\item by property (2) in the statement of the theorem, at most finitely many of the $ U_j $ have spherical diameters greater than $ \varepsilon $.
\end{enumerate}
It follows that $ \bigcup_{j\in \N} \overline{U}_j $ must be bounded.  If we now let
\[ D = C \cup \bigcup_{j \in \N} \overline{U}_j, \]
and put
\[ G = \text{int}(\widetilde{D}), \] 
we then have that $ G $ is a bounded, \sconn\ domain whose  boundary $ \partial G $ lies in $ J(f). $  

Now choose $ r' > r $ such that $ G \subset B(0,r') $, and let $ \varepsilon' > 0 $ be less than the spherical distance of the circle $ C' = C(0,r') $ from $ \infty $.  Then we may proceed exactly as above to obtain a  
bounded, \sconn\ domain $ G' \supset G $ whose  boundary $ \partial G' $ lies in $ J(f) $.
  
In this way, we may evidently construct a sequence $ (G_k)_{k \in \N} $ of bounded, \sconn\ domains such that $ G_{k+1} \supset G_k $ and $ \partial G_k \subset J(f) $ for each $ k \in \N $, and $ \bigcup_{k \in \N} G_k = \C $. 
This completes the proof.
\end{proof}

\begin{proof}[Proof of Corollary \ref{spider}]
To prove that $ J(f) $ is a \sw, it only remains to show that $ J(f) $ is connected.  But since there are no \mconn\ Fatou components, this is immediate from Lemma \ref{kisaka}.
\end{proof}

\section{Proof of Theorems \ref{rings} and \ref{burjor}} 
\label{Thm rings}
\setcounter{equation}{0} 

In this section, we first prove Theorem \ref{rings}, which says that, if $ f $ is a \tef\ such that $ J(f) $ is a \sw, then there exists a subset of $ J(f) $ which is dense in $ J(f) $ and consists of points $ z $ with the property that any \nhd\ of $ z $ contains a continuum in $ J(f) $ surrounding ~$ z $ and, furthermore, that each such point $ z $ is a buried point of $ J(f) $ at which $ J(f) $ is \lconn.  The method of proof is similar to that adopted by Bergweiler~\cite{wB00} in his alternative proof of a result due to Dom\'{i}nguez \cite{PD1}. 

We make use of the following corollary to the Ahlfors islands theorem, proved in ~\cite{wB00} for a wide class of meromorphic functions, but here stated in a form applicable to \tef s since this is all we will need.

\begin{proposition} \label{islands}
Let $ f $ be a \tef, and let $ D_1, D_2, D_3 \subset \C $ be bounded Jordan domains with pairwise disjoint closures.  Let $ V_1, V_2, V_3 $ be domains satisfying $ V_j \cap J(f) \neq \emptyset $ and $ V_j \subset D_j $ for $ j \in \lbrace 1, 2, 3 \rbrace. $ Then there exist $ \mu \in \lbrace 1, 2, 3 \rbrace $, $ n \in \N $ and a domain $ U \subset V_\mu $ such that $ f^n : U \rightarrow D_\mu $ is conformal.
\end{proposition}  

\begin{proof}[Proof of Theorem \ref{rings}]
Since $ J(f) $ is a \sw, it follows from the definition that we may choose a sequence $ (G_k)_{k \in \N} $ of bounded, \sconn\ domains with \textit{disjoint} boundaries $ \partial G_k, $ and such that
\begin{itemize}
\item $ G_{k+1} \supset G_k $, for $ k \in \N $,
\item $ \partial G_k \subset J(f) $, for $ k \in \N $ and 
\item $ \bigcup_{k \in \N} G_k = \C. $
\end{itemize}

Now let $ G $ be any domain in the sequence $ (G_k)_{k \in \N} $ that meets $ J(f) $, and let $ z \in G \cap J(f) $ be such that $ z \notin E(f) $. Then, by Picard's theorem, there are infinitely many preimages of $ z $ under $ f $, and each of these must lie in a component of $ f^{-1}(G) $.  Note that a component of $ f^{-1}(G) $ can in general be either bounded or unbounded, and can contain more than one preimage of $ z $ under $ f $.   

Let $ w_1 $ be some preimage of $ z $ under $ f $, and let $ W_1 $ be the component of $ f^{-1}(G) $ containing $ w_1 $.  Then we can assume that, for some $ k \in \N $,
\[ w_1 \in G_{k+2} \setminus \overline{G}_k. \]   
It follows that $ w_1 $ lies in a bounded domain $ V_1 $ which is a component of
\[ (G_{k+2} \setminus \overline{G}_k) \cap W_1.\]
Furthermore, $ V_1 \cap J(f) \neq \emptyset $ (because $ w_1 \in J(f) $) and, since the boundaries of both $ G_{k+2} \setminus \overline{G}_k $ and of $ W_1 $ lie in $ J(f) $, we also have $ \partial V_1 \subset J(f). $ 

Now since $ \overline{G}_{k+3} $ is bounded, it can contain only finitely many preimages of $ z $ and thus we may choose another preimage of $ z $ under $ f $, $ w_2 $ say, that lies outside $ \overline{G}_{k+3} $ and in some component $ W_2 $ of $ f^{-1}(G) $.  Proceeding exactly as before, we find that, for some $ k' \geq k+3 $, $ w_2 $ lies in a bounded domain $ V_2 $ which is a component of
\[ (G_{k'+2} \setminus \overline{G}_{k'}) \cap W_2.\]
We also have $ V_2 \cap J(f) \neq \emptyset $, and $ \partial V_2 \subset J(f). $  Note that $ W_2 $ is not necessarily distinct from $ W_1 $, but that, by construction, $ V_1 $ and $ V_2 $ have disjoint closures.

Continuing in the same way, we can evidently construct domains $ V_1, V_2, V_3 $ with pairwise disjoint closures such that, for $ j = 1,2,3, $
\begin{itemize}
\item $ V_j \cap J(f) \neq \emptyset $; and
\item $ \partial V_j \subset J(f) $.
\end{itemize}
Furthermore, it follows from \cite[Theorem 3.3, p. 143]{New} that we can then choose bounded, simply connected, Jordan domains $ D_1, D_2, D_3 $ with pairwise disjoint closures such that $ \overline{V}_j \subset D_j $ for $ j = 1, 2, 3 $.  

Everything is now in place for us to apply Proposition \ref{islands}, and thus we obtain $ \mu \in \lbrace 1, 2, 3 \rbrace $, $ n \in \N, $ and a domain $ U \subset V_\mu $ such that $ f^n : U \rightarrow D_\mu $ is conformal.  

Now let $ \phi $ be the branch of the inverse function $ f^{-n} $ which maps $ D_\mu $ onto $ U. $  Then $ \phi $ must have a fixed point $ z_0 \in U \subset V_\mu $.  Furthermore, by the Schwarz lemma, this fixed point must be attracting, and because $ \phi(D_\mu) = U $ where $ \overline{U} $ is a compact subset of $ D_\mu $, we have that $ \phi^k(z) \rightarrow z_0 $ as $ k \rightarrow \infty $, uniformly for $ z \in D_\mu $.

Since $ z_0 $ is an attracting fixed point of $ \phi $, it is a repelling fixed point of $ f^n $ and hence a repelling periodic point of $ f $.  Thus $ z_0 $ lies in $ J(f) $.  

Now $ z_0 = \phi^k(z_0) \in \phi^k(V_\mu)$ for all $ k \in \N. $  Furthermore, $ \text{diam } \phi^k(\overline{V_\mu}) \rightarrow 0 $ as $ k \rightarrow \infty $.  It follows that
\begin{equation*}
\label{single}
\bigcap_{k \in \N} \phi^k(\overline{V_\mu}) = \lbrace z_0 \rbrace, 
\end{equation*}
and hence that, for any \nhd\ $ N $ of $ z_0 $, there is some $ K \in \N $ such that $ \phi^K(\overline{V_\mu}) \subset N $. But $ \partial V_\mu $ lies in $ J(f) $ and $ \phi $ is conformal, so we have $ \partial \phi^K (V_\mu) = \phi^K (\partial V_\mu) \subset J(f) $, and since $ \partial \phi^K (V_\mu) $ surrounds $ z_0 $, we have shown that an arbitrary \nhd\ $ N $ of $ z_0 $ contains a continuum in $ J(f) $ that surrounds~$ z_0. $ 

To show that points with this property are dense in $ J(f) $, we use the fact that $ J(f) $ is the closure of the backwards orbit $ O^-(z) $ of any point $ z \in J(f) \setminus E(f) $. Now we may always choose our domains $ V_j $ to ensure that $ z_0 \notin E(f) $.  Therefore, since $ f $ is an open mapping and $ J(f) $ is completely invariant, it follows that each point $ z $ in the backwards orbit $ O^-(z_0) $ has the property that any \nhd\ of $ z $ contains a continuum in $ J(f) $ that surrounds~$ z $.  

Now let $ z $ be a point with this property. Evidently, $ z $ does not lie on the boundary of any component of $ F(f) $, and so is a buried point.  Let $ V $ be an open \nhd\ of $ z $ in the relative topology on $ J(f) $, so that $ V = V' \cap J(f) $ for some open \nhd\ $ V' $ of $ z $ in~$ \C $.  We may assume without loss of generality that $ V' $ is a disc (by making $ V $ smaller if necessary).  Then it follows from the assumed property of $ z $ that $ V' $ contains a continuum $ C $ in $ J(f) $ surrounding ~$ z $.  

Now let $ X = \widetilde{C} \cap J(f) $ (recall that $ \widetilde{C} $ denotes the union of $ C $ and its bounded complementary components).  Since $ J(f) $ is a \sw, it is connected, and it follows that $ X $ is also connected.  But $ \widetilde{C} \subset V' $, so $ X \subset V $, and thus we have shown that any \nhd\ $ V $ of $ z $ in the relative topology on $ J(f) $ contains a connected \nhd\ of $ z $.  It then follows from the definition that $ J(f) $ is \lconn\ at $ z $.  This completes the proof.
\end{proof}

\begin{remark}
In \cite[Theorem 1.6]{O10}, we showed that, if $ f $ is a \tef, $ R > 0 $ is such that $ M(r, f) > r $ for $ r \geq R $, and $ A_R(f) $ is a \sw, then $ J(f) $ has a dense subset of \textit{periodic} buried points (see Section \ref{prelim} for the definition of the set $ A_R(f) $).  We remark that, using a similar method of proof, it is possible to extend Theorem \ref{rings} to show that, if $ f $ is a \tef\ such that $ J(f) $ is a \sw, then there exists a dense subset of periodic buried points, at each of which $ J(f) $ is \lconn.  We omit the details.
\end{remark}

\begin{proof}[Proof of Theorem \ref{burjor}]
It is immediate from Theorems \ref{lcsw} and \ref{rings} that $ J(f) $ is a \sw\ and that $ J_r(f) \neq \emptyset. $ The rest of part (a) follows from Lemma \ref{whyburn2}.  For part (b), it remains to prove that the $ J(f) $ \sw\ contains a sequence of loops that are Jordan curves.

Let $ z $ be a buried point in $ J(f) $. Then, by part (a), there is a Jordan curve $ C $ in $ J(f) $ surrounding $ z $.  Now let $ G = \text{int} (\widetilde{C}) $, and let $ \gamma_n $ be the outer boundary component of $ f^n(G). $  Then, by the blowing up property of $ J(f) $,
\[ \mathrm{dist}(\gamma_n, 0) \to \infty \text{  as  } n \to \infty. \]
Since $ \gamma_n \subset f^n(C) $, it follows that $ \gamma_n $ is a Jordan curve in $ J(f) $.

Thus $ G_n = \mathrm{int}(\widetilde{\gamma_n}) $ is a bounded Jordan domain for each $ n \in \N $.  Furthermore, $ \partial G_n \subset J(f) $ for each $ n \in \N $, and we can choose a subsequence $ (G_{n_k})_{k \in \N} $ such that $ \bigcup_{k \in \N} G_{n_k} = \C, $ and $ G_{n_{k+1}} \supset G_{n_k} $ for $ k \in \N $.   It follows that, by relabelling $ G_{n_k} $ as $ G_k $ for $ k \in \N $, we obtain a sequence of bounded Jordan domains $ (G_k)_{k \in \N} $ with the required properties, and this completes the proof.
\end{proof}

\section{The residual Julia set} 
\label{residual}
\setcounter{equation}{0} 

In this section, we give some new results on the residual Julia set $ J_r(f) $ of a \tef\ $ f $, and compare the results on $ J_r(f) $ in Theorems \ref{rings} and \ref{burjor} with those obtained by other authors.

Recall that the residual Julia set $ J_r(f) $ of a map $ f $ is the set of buried points, i.e. the set of points in $ J(f) $ that do not lie on the boundary of any Fatou component.

First, we draw attention to a corollary of the following result due to Rippon and Stallard.

\begin{lemma}[Theorem 5.2 in \cite{RS11}] 
\label{rs}
Let $ f $ be a \tef, and suppose that the set $ S $ is completely invariant under $ f $ and that $ J(f) = \overline{S \cap J(f)}. $  Then exactly one of the following holds:
\begin{enumerate}[(1)]
\item $ S $ is connected;
\item $ S $ has exactly two components, one of which is a singleton $ \lbrace \alpha \rbrace $, where $ \alpha $ is a fixed point of $ f $ and $ \alpha \in E(f) \cap F(f) $;
\item $ S $ has infinitely many components.
\end{enumerate}
\end{lemma}

For the residual Julia set, we obtain the following.

\begin{corollary}
Let $ f $ be a \tef\ with non-empty residual Julia set $ J_r(f) $.  Then either $ J_r(f) $ is connected, or else $ J_r(f) $ has infinitely many components.
\end{corollary}

\begin{proof}
Since $ J_r(f) $ is completely invariant and dense in $ J(f) $, it is evident that the conditions of Lemma \ref{rs} hold with $ S = J_r(f) $.  Case (2) cannot occur since $ J_r(f) \cap F(f) = \emptyset. $ 
\end{proof}

Next, we show that there are certain classes of functions for which the residual Julia set is not only connected, but is in fact a \sw.  Our result is expressed in terms of the \textit{fast escaping set}, defined as follows:
\[ A(f) = \lbrace z \in \C : \text{there exists } \ell \in \N \text{ such that } \vert f^{n+ \ell}(z) \vert \geq M^n(R,f), \text{ for } n \in \N \rbrace.  \]
We refer to \cite{RS10a} for a detailed study of $ A(f) $ and references to earlier work.

\begin{theorem}
\label{residualsw}
Let $ f $ be a \tef, let $ R > 0 $ be such that $ M(r, f) > r $ for $ r \geq R $, and let $ A_R(f) $ be a \sw.  Assume also that $ A(f) \subset J(f). $  Then $ J_r(f) $ is a \sw.
\end{theorem}

\begin{proof}
Since $ A(f) \cap F(f) = \emptyset $, there are no \mconn\ Fatou components by \cite[Theorem 4.4]{RS10a}, so $ J(f) $ is a \sw\ by Lemma \ref{fast}.  Furthermore, no point on the boundary of a Fatou component of $ f $ can lie in $ A(f) $ by  
\cite[Theorem 1.1(a)]{O10}.
Thus
\[ A(f) \subset J_r(f) \subset J(f) = \overline{A(f)}, \]
because $ J(f) = \partial A(f) $ \cite{BH99}.  Since $ A(f) $ is a \sw\ by \cite[Theorem 1.4]{RS10a}, it follows that $ J_r(f) $ is connected and indeed is also a \sw. 
\end{proof}

An example of a class of functions for which $ J_r(f) $ is a \sw\ is Baker's construction \cite{iB01} of \tef s of arbitrarily small growth, for which every point in the Fatou set tends to a superattracting fixed point at $ 0 $ under iteration (independently, Boyd \cite{Boy} arrived at a very similar construction).  Clearly $ A(f) \subset J(f) $ for such functions, and it follows from \cite[Theorem 1.9(b)]{RS10a} that $ A_R(f) $ is a \sw.

We remark that, when $ J_r(f) $ is a \sw, we have the following analogue of Theorem \ref{rings} (the proof is very similar and we omit it).

\begin{theorem}
\label{residualrings}
Let $ f $ be a \tef\ such that $ J_r(f) $ is a \sw.  Then there exists a subset of $ J_r(f) $ which is dense in $ J(f) $ and consists of points $ z $ with the property that every \nhd\ of $ z $ contains a continuum in $ J_r(f) $ that surrounds ~$ z $.  At each such point, $ J_r(f) $ is \lconn.     
\end{theorem}
  
Finally in this section, we compare our results on $ J_r(f) $ in Theorems \ref{rings} and \ref{burjor} with those obtained by other authors.

Theorem \ref{rings} gives a sufficient condition for a \tef\ to have $ J_r(f) \neq \emptyset $, namely that $ J(f) $ is a \sw.  This complements other sufficient conditions in the literature for $ J_r(f) $ to be non-empty:
\begin{itemize}
\item Baker and Dom\'{i}nguez \cite[Theorem 6]{BD2} showed that $ J_r(f) \neq ~\emptyset $ if $ F(f) $ is not connected, there are no \wand s, and all periodic Fatou components are bounded;
\item Dom\'{i}nguez and Fagella \cite[Proposition 6.1]{DF} proved that, if all Fatou components eventually iterate inside a closed set $ A \subsetneq \C $ with non-empty interior and never leave again, then $ J_r(f) \neq \emptyset $ provided the complement of $ A $ meets $ J(f) $. 
\end{itemize}
 
Theorem \ref{burjor} gives us, in particular, that $ J_r(f) \neq \emptyset $ whenever $ f $ is a \tef\ such that $ J(f) $ is \lconn.  For a general \tef, this result appears to be new.  However, for \tef s in the class $ \mathcal{S} $, it is implied by a result of Ng, Zheng and Choi \cite[Theorem 2.1]{NZC}.

We remark that, for each of the examples given in Section \ref{examples} below, it is immediate from our results that the residual Julia set is not empty.  This has already been proved explicitly for some of the functions or classes of functions discussed - see, for example, \cite[Corollary 6.5]{DF}, \cite[Theorem 6]{Mor} and \cite[Proposition 7.1]{NZC}.

\section{Examples} 
\label{examples}
\setcounter{equation}{0} 

In this section we give a number of examples which illustrate the results of previous sections.   

First, we consider \tef s for which the Julia set is a \sw.  We describe a large class of such functions, based on the work of Rippon and Stallard in \cite{RS10a}.  We also show that the Julia set can be a \sw\ for functions outside this class, by proving that $ J(g) $ is a \sw\ when $ g(z) = \sin z $.  For each of these functions, it follows from Theorem \ref{rings} that the Julia set is \lconn\ at a dense subset of buried points. 

In \cite{RS10a}, Rippon and Stallard discussed the properties of the set $ A_R(f) $ defined in Section \ref{prelim}, and in \cite[Theorem 1.9]{RS10a} gave many examples of functions for which $ A_R(f) $ is a \sw.  These examples include functions with

\begin{myindentpar}{1cm}
\begin{enumerate}[(a)]
\item very small growth,
\item order less than $ \tfrac{1}{2} $ and regular growth,
\item finite order, Fabry gaps and regular growth, or
\item a sufficiently strong version of the pits effect, and regular growth.
\end{enumerate}
\end{myindentpar}

The terminology used here is defined and made precise in \cite{RS10a}.  Mihaljevi\'{c}-Brandt and Peter \cite{MP1}, and Sixsmith \cite{S1}, have given further classes of \tef s for which $ A_R(f) $ is a \sw.

For each of these functions, it follows from Lemma \ref{fast} that $ J(f) $ is a \sw\ whenever $ f $ has no \mconn\ Fatou components.  Note that the escaping set $ I(f) $ is also a \sw\ for these functions, by \cite[Theorem 1.4]{RS10a}. 

Now a \tef\ such that $ A_R(f) $ is a \sw\ can never belong to the class $ \mathcal{S} $ or the class $ \mathcal{B} $, by \cite[Theorem 1.8]{RS10a}.  However, $ J(f) $ can still be a \sw\ in these circumstances, as we now show.

The function $ g(z) = \sin{z} $ has been the subject of a number of studies \cite{BD1, PD1, DS}.  In particular,  Dom\'{i}nguez proved in \cite{PD1} that $ J(g) $ is connected, and Baker and  Dom\'{i}nguez showed in \cite{BD1} that $ J(g) $ is \lconn\ at the fixed point $ 0 $.  We now prove that $ J(g) $ is a \sw, and also show that neither the escaping set $ I(g) $ nor the residual Julia set $ J_r(g) $ is a \sw.

We will need the following result (see, for example, \cite[Theorem 7.9]{Con}). 

\begin{lemma}[part of the Koebe Distortion Theorem] 
\label{distort}
Let $ f $ be a function that is univalent on the unit disc with $ f(0) = 0 $ and $ f'(0) = 1 $.  Then, if $ \vert z \vert < 1 $,
\[ \vert f(z) \vert \leq \dfrac{\vert z \vert}{(1 - \vert z \vert)^2}.  \] 
\end{lemma}

\begin{example}
\label{one}
Let $ g(z) = \sin{z}.$  Then $ J(g) $ is a \sw, but neither $ I(g) $ nor $ J_r(g) $ is a \sw.
\end{example}

\begin{proof}
We first recall some basic facts about the Fatou components of $ g $ from \cite{BD1, PD1}. It is clear that $ g \in \mathcal{S} $, and that the singular values of $ g $ are the two critical values at $ \pm 1. $ The fixed point $ 0 $ lies on the boundary of two invariant, parabolic Fatou components, which are reflections of one another in the imaginary axis, and in each of which $ g^n(z) \to ~0 $ as $ n \to \infty. $  

Label these components $ D_0 $ and $ D_{-1} $, where $ D_0 $ meets the positive real axis and $ D_{-1} $ is its reflection in the imaginary axis. Then $ D_0 $ and $ D_{-1} $ are bounded, and are the only two periodic Fatou components, each containing the entire orbit of one of the critical values. Furthermore, every other Fatou component is a preimage of either $ D_0 $ or $ D_{-1} $ under $ g^m $ for some $ m \in \N $, and the components of $ g^{-1}(D_0) $ and $ g^{-1}(D_{-1}) $ all have the form
\[ D_n = \lbrace z + n \pi : z \in D_0, n \in \Z \rbrace. \]

We claim that the diameters of all of the components of $ F(g) $ are uniformly bounded. 

To prove the claim, we begin by using ideas from Baker and Dom\'{i}nguez' proof of Theorem F in \cite{BD1}.  There it is shown that, apart from the point $ 0 $, the lemniscate $ \vert z^2 - 1 \vert = 1 $ lies in $ D_0 \cup D_{-1} $. If $ h $ is any branch of $ g^{-1} $, a straightforward calculation therefore shows that $ \vert h'(z) \vert < 1 $ outside the lemniscate, and hence in any component of $ F(g) $ other than $ D_0 $ and $ D_{-1} $.

Now let $ U $ be any Fatou component of $ g $ other than a component of $ g^{-1}(D_0) $ or $ g^{-1}(D_{-1}) $.  Then there exists $ m \geq 1 $ and $ n \neq 0, -1 $ such that $ g^m(U) =  D_n. $  Furthermore, because the orbits of both critical values lie entirely in the real interval $ [-1, 1] $, the branch $ \phi $ of $ g^{-m} $ mapping $ D_n $ to $ U $ is univalent in some domain $ G $ containing $ \overline{D}_n $. Now no component of $ g^{-k}(D_n) $ for $ k \in \lbrace 1, \ldots, m \rbrace $ meets $ D_0 \cup D_{-1} $, since $ D_0 $ and $ D_{-1} $ are invariant.  Therefore, since $ \phi $ is a composition of branches $ h $ of $ g^{-1} $, for each of which $ \vert h'(z) \vert < 1 $ outside $ D_0 \cup D_{-1} $, it  follows that $ \vert \phi'(z) \vert < 1 $ throughout $ D_n. $   

Now, following ideas from the proof in \cite[Theorem 7.16]{Con}, let $ d $ be such that $ 0 < 2d < \mathrm{dist}(D_n, \partial G) $, and cover the compact set $ \overline{D}_n $ by a finite collection $ \mathbf{B} $ of open discs of radius $ d / 8, $ each of which meets  $ \overline{D}_n $.   Let $ B_1, B_2 $ be two discs from this collection with non-empty intersection, and let $ z_1 \in B_1 \cap D_n $ and $ z_2 \in B_2 \cap D_n $.  Then we have $ \vert z_1 - z_2 \vert < d/2 $, and $ B_1 \cup B_2 \subset \overline{B}(z_1, d) \subset G $.

Now the function
\[ \psi(z) = \dfrac{\phi(z_1 + dz) - \phi(z_1)}{d \phi'(z_1)} \]
is univalent in the unit disc, with $ \psi(0) = 0 $ and $ \psi'(0) = 1 $. Thus it follows from Lemma \ref{distort} that
\[ \left \vert \dfrac{\phi(z_1 + dz) - \phi(z_1)}{d \phi'(z_1)} \right \vert \leq \dfrac{\vert z \vert}{(1 - \vert z \vert)^2} \]
for $ \vert z \vert < 1. $  If we now put $ z = (z_2 - z_1) / d $, so that $ \vert z \vert < 1/2 $, and use the fact that $ \vert \phi'(z) \vert < 1 $ throughout $ D_n $, we obtain
\[ \vert \phi(z_2) - \phi(z_1) \vert \leq 2d. \]
Now let $ z, w $ be arbitrary points in $ D_n $.  Then there are points $ z = z_1, z_2, \ldots, z_k = w $ in $ D_n $, with $ z_i \in B_i \in \mathbf{B} $ for $ i = 1, \ldots, k $, where each consecutive pair of discs has non-empty intersection.  It follows that
\[ \vert \phi(z) - \phi(w) \vert \leq \sum_{j = 1}^{k - 1} \vert \phi(z_j) - \phi(z_{j + 1}) \vert \leq 2(k -1)d < 2Kd,  \]
where $ K $ is the total number of discs in $ \mathbf{B} $.  Thus the diameter of $ U $ is at most $ 2Kd. $
 
Furthermore, since the Fatou components $ D_n $ are congruent for all $ n \in \Z , n \neq 0, -1, $ we can use the same value of $ d $ and congruent open covers whatever the value of $ n $.  Since $ D_0 $ and $ D_{-1} $ are bounded, this completes the proof of the claim.
 
Now let $ \rho > 0, $ and let $ \lbrace U_j : j \in \N \rbrace $ be the collection of components of $ F(g) $ that meet the circle $ C(0, \rho). $  Then it follows from the claim just proved that the set $ \bigcup_{j \in \N} \overline{U}_j $ is bounded.  If we now let
\[ X = C(0, \rho) \cup \bigcup_{j \in \N} \overline{U}_j , \]
and put
\[ G = \text{int}(\widetilde{X}), \]
we then have that $ G $ is a bounded, \sconn\ domain whose boundary $ \partial G $ lies in $ J(g). $

We can now proceed exactly as in the proof of Theorem \ref{no unbounded}, and construct a sequence $ (G_k)_{k \in \N} $ of bounded, \sconn\ domains such that $ G_{k+1} \supset G_k $ and $ \partial G_k \subset J(g) $ for each $ k \in \N $, and $ \bigcup_{k \in \N} G_k = \C $.  Since we know that $ J(g) $ is connected, it follows that $ J(g) $ is a \sw.

Finally, we note that $ g $ maps the real line onto the interval $ [-1, 1] $, so that there are no points on the real line that escape to infinity under iteration.  Furthermore, all points on the real line are in the Fatou set, except for the points $ \lbrace z = n \pi : n \in \Z \rbrace $, which each lie on the boundaries of two adjacent Fatou components.  This shows that neither $ I(g) $ nor $ J_r(g) $ is a \sw.
\end{proof}

Recall that, by Theorem \ref{rings}, the Julia set for $ g(z) = \sin{z} $ is \lconn\ at a dense subset of buried points.  This adds to the result of Baker and Dom\'{i}nguez \cite{BD1} that $ J(g) $ is \lconn\ at the fixed point $  0 $ and its preimages (which are not buried points).  However, it seems to be an open question whether $ J(g) $ is everywhere \lconn.

We now briefly review the conditions under which it is known that a \tef\ has a \lconn\ Julia set and give a number of examples from the literature of functions with this property.  We also use results from the literature to derive some further examples.  For the functions in each of these examples, it follows from Theorem \ref{burjor} that the Julia set is a \sw\ containing a sequence of loops $ (\partial G_k)_{k \in \N} $ which are Jordan curves and which are the boundaries of a sequence of bounded, \sconn\ domains $ (G_k)_{k \in \N} $ satisfying (\ref{web}).

For rational maps, it has long been known that the local connectedness of the Julia set is related to the orbits of the critical points of the map (its \textit{critical orbits}).  A rational map $ R $ is \textit{hyperbolic} if the closure of the union of its critical orbits is disjoint from $ J(R) $ and, for such a map, if $ J(R) $ is connected then it is also \lconn.  The related, but weaker, concepts of subhyperbolic, semihyperbolic and geometrically finite rational maps have also been investigated, and for these maps too, if the Julia set is connected then it is \lconn.  We refer to \cite[Chapter 19]{Mil} and to \cite{CJY, MH, TY}.   

Attempts to extend these ideas to \tef s have had some success.  For example, the following result in this direction is a version of a theorem stated by Morosawa \cite[Theorem 2]{Mor}.  

\begin{lemma}
\label{mor}
Let $ f $ be a \tef\ in the class $ \mathcal{S} $ and such that each component of $ F(f) $ contains at most finitely many critical points.  Assume further that all cyclic components of $ F(f) $ are bounded. Then $ J(f) $ is \lconn\ if the following two conditions hold:
\begin{enumerate}[(1)]
\item if $ \zeta \in F(f) \cap \mathrm{sing}(f^{-1}) $, then $ \zeta $ is a critical value and is absorbed by an attracting cycle;
\item if $ \zeta \in J(f) \cap \mathrm{sing}(f^{-1}) $, then for any Fatou component $ D $ we have
\[ \overline{ \bigcup_{n \geq 0} f^n (\zeta) } \cap \partial D = \emptyset. \]
\end{enumerate}
\end{lemma}

\begin{remark}
In \cite[Theorem 2]{Mor} it was assumed only that $ f $ is in the class $ \mathcal{S} $, and the additional assumption in Lemma \ref{mor} that each component of $ F(f) $ contains at most finitely many critical points was omitted.  The proof of \cite[Theorem ~2]{Mor} requires the deduction that if the closure of a bounded component of $ F(f) $ contains no asymptotic value of $ f $, then all the components of its preimages are bounded.  The author is grateful to the referee for pointing out that this deduction requires a stronger hypothesis than that $ f $ is in the class $ \mathcal{S} $, since a preimage could contain infinitely many critical points (in which case it must be unbounded).  
\end{remark}
 
Using this result, Morosawa gave the following examples of \tef s in the class $ \mathcal{S} $ for which the Julia set is \lconn\ (note that in each of these examples the function has only one critical point):

\begin{itemize}
\item $ f_\lambda(z) = \lambda z e^z, $ where $ \lambda $ is such that $ f_\lambda $ has an attracting cycle whose period is greater than one, and satisfies $ \vert \mathrm{Im } (\lambda) \vert \geq e \mathrm{Arg } (\lambda) $ \cite[Theorem 5]{Mor}. \\
\item  $ g_a(z) = ae^a(z - (1 - a))e^z, $ where $ a > 1 $ \cite[Theorem 7]{Mor}.  
\end{itemize}

Indeed, Morosawa showed that $ J(g_a) $ is homeomorphic to the \textit{Sierpi\'{n}ski curve} continuum, i.e. that it is a nowhere dense subset of $ \widehat{\C} $ which is closed, connected and \lconn, and has the property that the boundaries of any two of its complementary components are disjoint Jordan curves \cite{Why2}.  It is a characteristic of the Sierpi\'{n}ski curve that it contains a homeomorphic copy of every one-dimensional plane continuum.  This was explored by Garijo, Jarque and Moreno Rocha \cite{GJR}, who have made a detailed study of the function $ g_a $, and demonstrated the existence of indecomposable continua in its Julia set.   

We note that, whenever the Julia set of a \tef\ is homeomorphic to the Sierpi\'{n}ski curve, it must necessarily also be a \sw\ by Theorem \ref{lcsw}. 

We now use Lemma \ref{mor} to give the following additional example of a \tef\ in the class $ \mathcal{S} $ for which the Julia set is \lconn.  The example is based on work by Dom\'{i}nguez and Fagella \cite{DF}, though they did not discuss local connectedness.   

\begin{example}
\label{two}
Let $ f(z) = \lambda \sin z, $ where $ \lambda \in \C $ is chosen so that there are two attracting cycles and is such that $ \vert \mathrm{Re } (\lambda) \vert \geq \tfrac{\pi}{2}. $ Then $ J(f) $ is \lconn. 
\end{example}

\begin{proof}  
It is shown in \cite[Proposition 6.3]{DF} that all the Fatou components of $ f $ are bounded (note that each Fatou component contains at most one critical point).  Clearly $ f \in \mathcal{S} $, and the singular values of $ f $ are the two critical values $ \pm \lambda. $   By the choice of $ \lambda $, each critical value is absorbed by an attracting cycle and it follows that $ J(f) \cap \text{sing}(f^{-1}) = \emptyset. $  Thus conditions (1) and (2) in Lemma \ref{mor} hold. 
\end{proof}

Under certain conditions, the Julia set is also \lconn\ for the class of semihyperbolic entire functions investigated by Bergweiler and Morosawa in \cite{BM02}.  

A \tef\ $ f $ is \textit{semihyperbolic} at $ a \in J(f) $ if there exist $ r > 0 $ and $ N \in \N $ such that, for all $ n \in \N $ and all components $ U $ of $ f^{-n}(B(a,r)) $, the function $ f^n \vert_U : U \to B(a,r) $ is a proper map of degree at most $ N. $  

Bergweiler and Morosawa's result on local connectedness is the following.

\begin{lemma}[Theorem 4 in \cite{BM02}]
\label{bergmor}
Let $ f $ be entire.  Assume that $ F(f) $ consists of finitely many attracting basins.  Suppose that if $ U $ is an immediate attracting basin, then $ U $ is bounded, $ f $ is semihyperbolic on $ \partial U $, and there exists $ N \in \N $ such that for every $ n \in \N $ and for every component $ V \neq U $ of $ f^{-n}(U) \setminus \bigcup_{k = 0}^{n - 1} f^{-k}(U) $ we have $ \mathrm{deg}(f^n \vert_V : V \to U) \leq N. $  Then $ J(f) $ is \lconn.
\end{lemma}

Using this result, Bergweiler and Morosawa gave the following example of a \tef\ with a locally connected Julia set which is in the class $ \mathcal{B} $ but not in the class $ \mathcal{S} $, i.e. the set $ \text{sing}(f^{-1}) $ is bounded but infinite. 

\begin{itemize}
\item There exists $ A $ such that, if $ \pi^2 < a < A, $ and
\[ f(z) = \dfrac{az}{\pi^2 - 4z} \cos{\sqrt{z}}, \]
 then $ f $ has an attracting fixed point such that $ F(f) $ consists of its basin, and the other conditions of Lemma \ref{bergmor} also hold \cite[Example 2]{BM02}.
\end{itemize}

We have now seen examples of functions in both $ \mathcal{S} $ and in $ \mathcal{B} \setminus \mathcal{S} $ which have \lconn\ Julia sets.  It is natural to ask for an example of a \tef\ $ f $ for which $ A_R(f) $ is a \sw\ (so that $ f $ is in neither $ \mathcal{S} $ nor $ \mathcal{B} $) and $ J(f) \neq \C $ is \lconn.  We end this paper by using Lemma \ref{bergmor} to give such an example.  

\begin{example}
\label{three}
Let $ f $ be in the class of \tef s of arbitrarily small growth constructed by Baker in \cite{iB01}, for which every point in the Fatou set tends to a superattracting fixed point at $ 0 $ under iteration.  Let $ R>0 $ be such that $ M(r,f) > r $ for $ r \geq R. $  Then $ A_R(f) $ is a \sw\ and $ J(f) $ is \lconn. 
\end{example}

\begin{proof}  
It follows from \cite[Theorem 1.9(b)]{RS10a} that $ A_R(f) $ is a \sw.  Furthermore, it is shown in \cite{iB01} that
\begin{enumerate}
\item each component of $ F(f) $ is bounded,
\item $ f^n(z) \to 0  $ as $ n \to \infty $, for each $ z \in F(f) $,
\item $ f $ has no finite asymptotic values, and
\item each of the critical points of $ f $, other than $ 0 $, lies in the escaping set $ I(f) $.
\end{enumerate} 

Let $ P(f) $ be the postcritical set of $ f $, that is
\[ P(f) = \lbrace f^n(\zeta) : \zeta \text{ is a critical value of } f, n \geq 0 \rbrace. \]
Then it can be shown using Baker's construction that, if $ U_0 $ is the immediate basin of the superattracting fixed point $ 0 $, there is a \nhd\ $ G $ of $ U_0 $ such that $ \overline{P(f)} \cap G = \lbrace 0 \rbrace $, and moreover that if $ U \neq U_0 $ is any other component of $ F(f) $, there is a \nhd\ $ G' $ of $ U $ such that $ \overline{P(f)} \cap G' = \emptyset. $  We omit the details.  

It follows in particular that
\begin{itemize}
\item $ f $ is semihyperbolic on $ \partial U_0 $, and
\item for each Fatou component $ U $, there exists $ n \in \N $ such that $ f^n(U) = U_0 $ and $ f^n \vert_U : U \to U_0 $ is univalent.
\end{itemize} 
Since $ F(f) $ consists of a single attracting basin, and $ U_0 $ is bounded, the conditions of Lemma \ref{bergmor} are satisfied.  Thus $ J(f) $ is \lconn.
\end{proof}

\begin{remark}
An alternative approach to proving the local connectedness of $ J(f) $ in Example \ref{three} would be to use Lemma \ref{whyburn}.  It follows from \cite[Theorem 1.5]{O10} that the boundary of every Fatou component of $ f $ is a Jordan curve, and a distortion argument can be used to show that, for each $ \varepsilon > 0 $, at most finitely many Fatou components have spherical diameters greater than $ \varepsilon $.
\end{remark}

\textbf{Acknowledgements} The author wishes to express his thanks to his doctoral supervisors, Prof P.J. Rippon and Prof G.M. Stallard, for their constant encouragement and especially for their assistance in the preparation of this paper, and to the referee for helpful comments which led to a number of improvements.

\end{document}